\documentclass[12pt, a4paper]{amsart}
\usepackage[centering]{geometry}
\usepackage{amssymb}
\usepackage{mathrsfs}
\usepackage{hyperref}
\usepackage{graphicx}
\usepackage{enumerate}
\usepackage[all,cmtip]{xy}
\usepackage{amscd}

\numberwithin{equation}{section}
\usepackage{amsthm}
\newtheorem{thm}{Theorem}[section]
\newtheorem{prop}{Proposition}[section]

\newtheorem{lemma}{Lemma}[section]
\newtheorem{conj}{Conjecture}[section]
\newtheorem{defn}{Definition}[section]

\theoremstyle{definition}

\title{Koebe uniformization for infinitely connected attracting Fatou domains} 
\author{Xiaoguang Wang}
\author{Yi Zhong}
\address{School of Mathematics, Zhejiang University, Hangzhou, 310027, China}
\email{wxg688@163.com}
\address{Institute of Mathematical Sciences and Applications, NingboTech University, Ningbo, 315100, China}
\email{yizhong@zju.edu.cn}
\date{\today}

\makeatletter
\@namedef{subjclassname@2020}{%
  \textup{2020} Mathematics Subject Classification}
\makeatother


\begin{document}

\begin{abstract}
This paper works on the structure of infinitely connected Fatou damains of rational maps in terms of Koebe uniformization. 
Due to the complicated boundary behavior, the existing uniformization results are failed to apply in general. 
We proved that if the rational map is geometrically finite, 
then its infinitely connected attracting Fatou damain is conformally homeomorphic to a circle domain.
\end{abstract}

\maketitle 

\setcounter{tocdepth}{1}
\tableofcontents

\section{Introduction}\label{sec1}
A circle domain  is a domain in the Riemann sphere $\widehat{\mathbb{C}}$ such that its boundary components are round circles or points.
In 1909, P. Koebe \cite{Koebe1} predicted that
\begin{conj}[Koebe's conjecture]
	Every connected domain in $\widehat{\mathbb{C}}$ is conformally homeomorphic to a circle domain.
\end{conj}
This conjecture was confirmed for finitely connected case by P. Koebe himself \cite{Koebe2}.
Later, P. Koebe \cite{Koebe3} proved that the conjecture is true for a class of domains with some symmetry. For domains with various conditions on the ``limit boundary components'', one can see \cite{Bers,Denneberg,Grotzsch,Haas,Meschowski1,Meschowski2,Sario,Sterbel1,Sterbel2}.

In 1993, Z.X. He and O. Schramm \cite{H-S1} proved that every countably connected domain admits a conformal map onto a circle domain, which is a major breakthrough.
Soon after O. Schramm \cite{Schramm} introduced the tool of transboundary extremal length, 
and utilized it to establish the uniformization for uncountablly connected ``cofat domains'':  
\begin{thm}\label{cofat}
	Every cofat domain in $\widehat{\mathbb{C}}$ is conformally homeomorphic to a circle domain.
\end{thm}
Here is the definition of the cofat domain.
\begin{defn}[\cite{Schramm}]\label{fat}
	A set $A\subset\widehat{\mathbb{C}}$ will be called $\tau-$fat, if for every $x\in A\cap\mathbb{C}$ and for every disk $B=B(x,r)$ of radius $r$ centered at $x$ that does not contain $A$ we have $\operatorname{area}(A\cap B)\geq\tau\cdot\operatorname{area}(B)$. A domain $\Omega\subset\widehat{\mathbb{C}}$ is cofat, if each connected component of its complement is $\tau-$fat for some $\tau>0$.
\end{defn}
In particular, a domain whose boundary components are uniform quasicircles or points is a cofat domain by the fatness of each quasidisk.

\vskip0.15in
Resently, the transboundary extremal length(or transboundary modulus) has been applied in the study of the uniformization of metric spaces, see \cite{Bonk1,B-M1,B-M2,H-L,M-W,R-R,Rehmert}. 
More results related to the Koebe uniformization problem can be found within the works that \cite{Bonk2,H-M,H-S2,H-S3,N-Y,Younsi}. 

Despite these developments, the general case of Koebe's conjecture remains unsolved and stands as an open problem.

\vskip0.15in
Fatou domains as a typical class of domains arising in the study of complex dynamical systems are meaningful to be associated to the Koebe uniformization problem.
For any rational map $f$ on $\widehat{\mathbb{C}}$, the well-known no wandering domains theorem \cite{Sullivan} tells us that all its Fatou components are eventually periodic. 
Moreover, there are only four types of periodic components, namely, attracting(superattracting), parabolic, Siegel disk and Herman ring. 
It is well known that a Siegel disk(or a Herman ring) is conformally equivalent to the unit disc(or a circular annulus). 
However, the uniformizations of attracting and parabolic Fatou domains are generally much more complicated. 
Since the uncountablely many and irregular boundary components, 
uniformization for such Fatou domains can not be covered by the existing works.

\vskip0.15in
In \cite{C-P}, Cui and Peng proved that any multiply-connected fixed attracting Fatou domain $U$ of a rational map $f$ is conformally homeomorphic to a domain $V$ such that its boundary component is either a quasicircle or a point. 
Unfortunately, those quasicircles are not uniform in general. Hence it can not be uniformized by Theorem \ref{cofat}.
The present work investigates the turning of such quasicircles to prove that the non-trivial boundary components of $V$ are uniform quasicircles when $f$ is a geometrically finite rational map. This leads to our main result:
\begin{thm}[Main theorem]\label{Fatou-Koebe}
	Let $U\subset\widehat{\mathbb{C}}$ be an infinitely connected attracting Fatou domain of a geometrically finite rational map $f$. 
	Then $U$ is conformally homeomorphic to a circle domain.
\end{thm}

\vskip0.15in
This paper is organized as follows. In section \ref{sec2}, we briefly recall some basic definitions and theorems of complex analysis and complex dynamics. Section \ref{sec3} is devoted to construct the puzzle pieces of Fatou domains $U$ and $V$. We show that the iterations of rational maps we studied have bounded degree property on the puzzle pieces in section \ref{sec4}. Finally, we proved that the Fatou domain $V$ is cofat by discussing the turning of its non-trivial boundary components, which is arranged in section \ref{sec5}.

\section{Preliminaries}\label{sec2}
This section presents some basic definitions and theorems of complex analysis and complex dynamics as a preparation for further discussions.

For a rational map $f$, denote $C(f)$ by the set of all critical points, 
$$P(f)=\overline{\bigcup_{n\geq1}f^n(C(f))}$$
by the postcritical set. Let $J_f$ be the Julia set of $f$, one of the definitions of geometrically finite rational map is:
\begin{defn}[\cite{T-Y}]
	The rational map $f$ is called geometrically finite if $P_f\cap J_f$ is finite.
\end{defn}

A Julia component of $f$ refers to a connected component within the Julia set. 
Since $f$ is continuous, for each Julia component $E$, we have $f(E)$ is also a Julia component. 
A Julia component $E$ of $f$ is considered periodic if $f^p(E)=E$ for some $p\geq1$, 
eventually periodic if $f^k(E)$ is periodic for some $k\geq0$, 
and wandering if neither of these conditions apply.
A Julia component is trivial if it is a single point, and non-trivial otherwise.  

\subsection{Proper map and turning distortion}
\begin{defn}[\cite{McMullen}]
	Let $U,V$ be a pair of Jordan disks. 
	A proper map between disks $f:U\rightarrow V$ is a holomorphic map such that $f^{-1}(K)$ is compact for every compact set $K\subset V$.
\end{defn}
Let $K$ be a connected and compact subset of $\mathbb{C}$ containing at least two points. 
For any $z_1,z_2\in K$, define the turning of $K$ about $z_1$ and $z_2$ by:
$$\Delta(K;z_1,z_2)=\frac{\operatorname{diam}(K)}{|z_1-z_2|},$$
where $\operatorname{diam}(\cdot)$ is the Euclidean diameter. 
Clearly, $1\leq\Delta(K;z_1,z_2)\leq\infty$ and $\Delta(K;z_1,z_2)=\infty$ if and only if $z_1=z_2$.

\begin{thm}[Turning distortion \cite{Q-W-Y}]\label{turning-distortion}
	Let $(V_j,U_j)$, $j\in\{1,2\}$ be a pair of simply connected planar domains 
	with $\overline{U}_j\subset V_j\subset\mathbb{C}$. 
	$g:V_1\rightarrow V_2$ is a proper holomorphic map of degree $d$, 
	$U_1$ is a connected component of $g^{-1}(U_2)$. 
	We assume $\operatorname{mod}(V_2\setminus\overline{U}_2)\geq m>0$, 
	then there is a constant $D(d,m)>0$ such that 
	$$\Delta(K;z_1,z_2)\leq D(d,m)\Delta(g(K);g(z_1),g(z_2))$$
	for any connected and compact subset $K$ of $U_1$ with $\#K\geq2$ and any $z_1,z_2\in K$.
\end{thm}

\subsection{Ahlfors' characterization of quasicircle}
A quasicircle is a Jordan curve in the complex plane that is the image of a circle under a quasiconformal map of the plane onto itself. It is called a $K-$quasicircle if the quasiconformal map has dilatation $K$. The interior of a quasicircle is called a quasidisk. In this paper, we will use an equivalent definition from Ahlfors \cite{Ahlfors}:
\begin{thm}\label{Ahlfors}
	A planar Jordan curve $S\subset\mathbb{C}$ is a K-quasicircle if and only if it satisfies the $K-$bounded turning condition: there is a constant $K\geq1$ such that for any $p,q\in S$, we have
	$$\Delta(L(p,q);p,q)\leq K,$$
	where $L(p,q)$ is the subarc of $S$ joining $p$ and $q$ with smaller diameter, i.e.
	$$\operatorname{diam}(L(p,q))\leq\operatorname{diam}(S\setminus L(p,q)).$$
\end{thm}

\section{Partition of Fatou domains}\label{sec3}
Let $U$ be a multiply-connected fixed attracting Fatou domain of a geometrically finite rational map $f$. 
Without loss of generality, we assume that $\infty\in U$ is the fixed point of $f$. 

Denote $\mathcal{E}$ by the collection of Julia components of $f$.
For each $E\in\mathcal{E}$, let $\widehat{E}$ be the union of $E$ and the connected component of $\widehat{\mathbb{C}}-E$ which is disjoint from $U$.
Note that $U$ is $f-$ invariant, but not necessarily complete, 
then we have either $f(\widehat{E})=\widehat{f(E)}$ or $f(\widehat{E})=\widehat{\mathbb{C}}$. 
Let 
$$\widetilde{\mathcal{E}}=\{E\in\mathcal{E}: f(\widehat{E})\neq\widehat{f(E)}\}.$$
Given $z\in\widehat{\mathbb{C}}$. 
The number of $f^{-1}(z)$ lies in $\widehat{\mathbb{C}}-U$(counted with multiflicity) is $\deg_{\widehat{\mathbb{C}}}(f)-d$. Then $\#\widetilde{\mathcal{E}}\leq\deg_{\widehat{\mathbb{C}}}(f)-d$, which means $\widetilde{\mathcal{E}}$ is a finite set. 

\subsection{Partition of $U$}
Let $C(f,U)$ be the set of critical points of $f$ in $U$ and $$P_f(U)=\overline{\bigcup_{n\geq1}f^n(C(f,U))}.$$
By the local theory of fixed point, there is a Jordan domain $U_0\subset U$ with smooth boundary such that
\begin{itemize}
	\item $\infty\in U_0\subset\overline{U}_0\subset f^{-1}(U_0)$;
	\item $\partial U_0\cap P_f(U)=\emptyset$;
	\item for any compact set $K\subset U$, there is an integer $k$ such that $g^k(K)\subset U_0$. 
\end{itemize}
We assume that $\partial U_0$ is a Jordan curve. 
Let $U_n$ be the connected component of $f^{-n}(U_0)$ that contains $U_0$. It is clear that $\overline{U}_n\subset U_{n+1}$ and $f(U_{n+1})=U_n$. Denote $\Gamma_n$ by the collection of the boundary components of $\partial U_n$, which consists of finitely many smooth Jordan curves.
For each $\gamma_n\in\Gamma_n$, let $P(\gamma_n)$ be the connected component of $\widehat{\mathbb{C}}-\gamma_n$ which is disjoint from $\infty$.
We call $\mathcal{P}_n=\{P(\gamma_n): \gamma_n\in\Gamma_n\}$ the collection of puzzle pieces of depth $n$. Considering two puzzle pieces of different depth $P_1\in\mathcal{P}_{n_1}$, $P_2\in\mathcal{P}_{n_2}$, only one of the following three will happen:
$$\overline{P}_1\subset P_2,\ \overline{P}_2\subset P_1,\ \overline{P}_1\cap\overline{P}_2=\emptyset.$$

\vskip0.15in
Since $U=\cup_{n\geq0}U_n,$
for each $E\in\mathcal{E}$ and each integer $n\geq0$, 
there exists an unique Jordan curve in $\Gamma_n$ that separates $E$ from $\infty$. Denote such curve by $\gamma_n(E)$. 
It is clear that $\gamma_{n+1}(E)\subset P(\gamma_n(E))$. 

For two different components $E_1$ and $E_2$, we may have $\gamma_n(E_1)=\gamma_n(E_2)$. This means $E_1$ and $E_2$ belong to the same puzzle piece $P(\gamma_n(E_1))$. 
However, as the depth $n$ tends to infinity, $E_1$ and $E_2$ will fall into two distinct puzzles.
Moreover, we have
$$\widehat{E}=\bigcap_{n\geq0}P(\gamma_n(E)).$$
\subsection{Partition of $V$}
	According to \cite{C-P}, there is a rational map $g$ and a completely invariant Fatou domain $V$ of $g$, such that
\begin{enumerate}
	\item $(f,U)$ and $(g,V)$ are holomorphically conjugate, 
	i.e. there is a conformal map $\varphi$ such that $\varphi\circ f=g\circ\varphi$ on $U$;
	\item each non-trivial boundary component of $V$ is a quasicircle which bounds an eventually superattracting Fatou domain of $g$ containing at most one postcritical point of $g$. 
	Moreover, $g$ is unique up to a holomorphic conjugation.
\end{enumerate}

Here $(g,V)$ is called a holomorphic model of $(f,U)$.
We now construct the puzzle pieces of $V$.
Since $V$ is $g-$completely invariant, we have $J_g=\partial V$. 
Denote $\mathcal{Q}$ by the collection of Julia components of $g$. 
For any $Q\in\mathcal{Q}$, we have $g(Q)\in\mathcal{Q}$ and $g(\widehat{Q})=\widehat{g(Q)}$. 
Let $V_n=\varphi(U_n)$, for each connected component $\gamma_n$ of $\partial U_n$, 
we denote $\beta_n=\varphi(\gamma_n)$ by the connected components of $\partial V_n$.
Set $I(\beta_n)$ as the bounded component of $\widehat{\mathbb{C}}-\beta_n$. 
The set $\mathcal{I}_n=\{I(\beta_n):\beta_n=\varphi(\gamma_n),\gamma_n\in\Gamma_n\}$ is the collection of all puzzle pieces of depth $n$.

\vskip0.15in
Since $E=\partial\left(\bigcap_{n\geq1}P(\gamma_n(E))\right)$,
we define $\Psi:\mathcal{E}\rightarrow\mathcal{Q}$ in the following way: 
\begin{equation}\label{Psi}
	\Psi(E)=\Psi\left(\partial\left(\bigcap_{n\geq1}P(\gamma_n(E))\right)\right)=\partial\left(\bigcap_{n\geq1}I(\varphi(\gamma_n(E)))\right)=Q.
\end{equation}
Therefore, the Julia components of $f$ and $g$ are one-to-one through the map $\Psi$.

\section{Bounded degree on puzzle pieces}\label{sec4}
Denote the collection of periodic Julia components of $f$ by $\mathcal{E}_{per}$.
\begin{lemma}\label{f-deg}
	Let $E\in\mathcal{E}\setminus\mathcal{E}_{per}$ be the eventually periodic Julia component of the geometrically finite rational map $f$, with $f^{p}(E)\in\mathcal{E}_{per}$ and $f^{p-1}(E)\notin\mathcal{E}_{per}$
	for some positive integer $p$. 
	Then there exists a sufficiently large integer $N>0$, such that  
	$$\deg\big(f^{p}:P\left(\gamma_{n+p}(E)\right)\rightarrow P\left(\gamma_{n}(f^p(E))\right)\big)$$
	is bounded by some constant $D$, which is independent of $p$ and $E$, for all $n>N$.
\end{lemma}
\begin{proof}
	Since $\widetilde{\mathcal{E}}$ and the critical points of $f$ in $U$ are both finite sets,  
	there exists some integer $N_0>0$ such that the sequence
	\begin{equation}\label{Puzzle-f}
		P(\gamma_{n+p}(E)),P(\gamma_{n+p-1}f(E)),\cdots,P(\gamma_{n+1}(f^{p-1}(E)))
	\end{equation}
    do not contain any components in $\widetilde{\mathcal{E}}$ or any critical points of $f$ in $U$ for all $n>N_0$.
	Then we only need to deal with the critical points in $\widehat{\mathbb{C}}-U$. Let $F_f$ be the Fatou set of $f$.
	
	{\bf Case I: The critical points in $F_f-U$.}\par
	Note that $E$ is strictly  eventually periodic, so the critical points in periodic Fatou components will not be contained in the sequence (\ref{Puzzle-f}). 
	On the other hand, each critical point in a strictly eventually periodic Fatou component will lie in some periodic Fatou component within finite iterations under $f$, 
    hence its critical orbit is a finite set outside the periodic Fatou components. 
	Then there exists an integer $N_1>0$, the sequence (\ref{Puzzle-f}) only meets such critical points at most once for the depth $n>N_1$.
	
	{\bf Case II: The critical points in $J_f$.}\par 
	Since $P_f\cap J_f$ is a finite set due to $f$ is geometrically finite, the critical points in $J_f$ are all eventually periodic. 
	We let $\mathcal{E}_{pc}$ be the collection of Julia components containing the points in $P_f$. 
	It is possible that the sequence (\ref{Puzzle-f}) may contain the components in $\mathcal{E}_{pc}$.
	With no loss of generality, we assume that 
	$$E_{pc}\subset P(\gamma_{n+p-j_0}(f^{j_0}(E)))$$
	for some integer $0<j_0<p$ and some $E_{pc}\in\mathcal{E}_{pc}$.
	This means $E_{pc}$ and $f^{j_0}(E)$ are in the same puzzle piece, which is  $$P(\gamma_{n+p-j_0}(E_{pc}))=P(\gamma_{n+p-j_0}(f^{j_0}(E))).$$ 
	Then we have
	$$P(\gamma_{n+p-j}(f^{j-j_0}(E_{pc})))=P(\gamma_{n+p-j}(f^{j}(E)))$$
	for $j>j_0$.
	
	Note that $f^{j-j_0}(E_{pc})\in\mathcal{E}_{pc}$ and $\mathcal{E}_{pc}$ is a finite set, we can choose a positive integer $N_2$ such that the puzzle pieces of the Julia components in $\mathcal{E}_{pc}$ are disjoint from each other for depth more than $N_2$. 
	Hence we have
	$$P(\gamma_{n+p-j_0}(E_{pc}))\cap P(\gamma_{n+p-j}(f^{j-j_0}(E_{pc})))=\emptyset$$
    for all $j_0<j<p$ and $n>N_2$.
    This implies
	\begin{equation}\label{pf}
		P(\gamma_{n+p-j_0}(E_{pc}))\cap P(\gamma_{n+p-j}(f^j(E)))=\emptyset,\ \ j_0<j<p,\ n>N_2.
	\end{equation}
	Due to the arbitrariness of $j_0$, (\ref{pf}) shows that the sequence (\ref{Puzzle-f}) can only intersect each component in $\mathcal{E}_{pc}$ at most once for all $n>N_2$.
	
	We take $N=\max\{N_0,N_1,N_2\}$. Then the number of critical puzzle pieces in the sequence (\ref{Puzzle-f}) is uniformly bounded for any positive integer $p$ and for any $n>N$. This implies the degree
	$$\deg\big(f^{p}:P\left(\gamma_{n+p}(E)\right)\rightarrow P\left(\gamma_{n}(f^p(E))\right)\big)$$
	is bounded by some constant $D$, which is independent of $p$ and the choice of $E$.
	Thus we have finished the proof.
\end{proof}

\begin{prop}\label{f-g}
    Let $E$ be a Julia component of $f$, then $Q=\Psi(E)$ is a Julia component of $g$. We have
	\begin{itemize}
		\item if $E$ is a $p-$periodic component of $f$, 
		then $Q$ is a $p-$periodic component of $g$;
		\item if $E$ is an eventually periodic component of $f$, 
		then $Q$ is an eventually periodic component of $g$. 
		Moreover, there is a positive integer $m$ such that $f^m(E)$ and $g^m(Q)$ are both periodic components, but $f^{m-1}(E)$ and $g^{m-1}(Q)$ are not;
		\item if $E$ is a wandering component $f$, 
		then $Q$ is a wandering component of $g$.
	\end{itemize}
\end{prop}
\begin{proof}
	Recall that $(f,U)$ and $(g,V)$ are holomorphic conjugate, 
	i.e. $\varphi\circ f=g\circ\varphi$ on $U$, we have
	$$g^p(\varphi(\gamma_n(E)))=\varphi(f^p(\gamma_n(E)))=\varphi(\gamma_{n-p}(f^p(E))).$$
	Then
	\begin{equation}\label{p-q}
		\partial\left(\bigcap_{n>p}I(g^p(\varphi(\gamma_{n}(E))))\right)=\partial\left(\bigcap_{n>p}I(\varphi(\gamma_{n-p}(f^p(E))))\right).
	\end{equation}
	By (\ref{Psi}), the right side of (\ref{p-q}) can be written into
	\begin{equation}
		\partial\left(\bigcap_{n>p}I(\varphi(\gamma_{n-p}(f^p(E))))\right)
		=\Psi\left(\partial\left(\bigcap_{n>p}P(\gamma_{n-p}(f^p(E)))\right)\right)
		=\Psi(f^p(E)).
	\end{equation}
	Since $\varphi(\gamma_n(E))=\beta_n(Q)$, the left side of (\ref{p-q}) can be transformed to
	\begin{equation}
		\begin{aligned}
			\partial\left(\bigcap_{n>p}I(g^p(\varphi(\gamma_{n}(E))))\right)
			&=\partial\left(\bigcap_{n>p}I(g^p(\beta_{n}(Q)))\right)\\
			&=\partial\left(\bigcap_{n>p}I(\beta_{n-p}(g^p(Q)))\right)\\
			&=g^p(Q).
		\end{aligned}
	\end{equation}
	Thus we have $\Psi(f^p(E))=g^p(Q)$. 
	Note that $E$ is $p-$periodic, then
	\begin{equation}\label{p}
		g^p(Q)=\Psi(f^p(E))=\Psi(E)=Q.
    \end{equation}
	The above process can be reversed due to $\varphi$ is conformal from $U$ to $V$. 
	Hence $p$ is the smallest integer that satisfying (\ref{p}), 
	which implies $Q$ is a $p-$periodic Julia component of $g$.
	
	Furthermore, we have $g^m(Q)=\Psi(f^m(E))$ for any positive integer $m$. 
	The rest part of proof is natural.
\end{proof}

\vskip0.15in

We now apply Lemma \ref{f-deg} and Proposition \ref{f-g} to estimate the degrees of the iterations of $g$. 
Since $\deg_{\gamma_n(E)}(f)\geq\deg_{\gamma_{n+1}(E)}(f)$ for each $E\in\mathcal{E}$, 
we have $\deg_{\gamma_n(E)}(f)$ converges to an integer, denoted by $\deg_{E}(f)$, as $n\rightarrow\infty$.
Similarly, $\deg_{\beta_n(Q)}(g)$ converges to an integer for each $Q\in\mathcal{Q}$, denoted by $\deg_{Q}(g)$, as $n\rightarrow\infty$. 
Moreover, \cite{C-P} shows that
$$\deg_E(f)=\deg_{\Psi(E)}(g).$$

\begin{lemma}\label{g-deg}
	Let $\mathcal{Q}_{per}$ be the collection of the periodic Julia components of $g$.
	Suppose that $Q$ is an eventually periodic Julia component of $g$, with
	$g^{p}(Q)\in\mathcal{Q}_{per}$ and $g^{p-1}(Q)\notin\mathcal{Q}_{per}$	
	for some positive integer $p$.	
	Then there exists a sufficiently large positive integer $N$ such that
	$$\deg(g^{p}:I(\beta_{n+p}(Q))\rightarrow I(\beta_{n}(g^{p}(Q))))$$
	is bounded by some constant $D$, which is independent of $p$ and $Q$, for all $n>N$.
\end{lemma}	
\begin{proof}
	Denote $\widetilde{\mathcal{Q}}=\{\Psi(E):\ E\in\widetilde{\mathcal{E}}\}$.
	
    {\bf Case I: Let $Q\in\mathcal{Q}-\widetilde{\mathcal{Q}}$ be eventually periodic.}
	Suppose that $Q=\Psi(E)$ for some eventually periodic component $E\in\mathcal{E}-\widetilde{\mathcal{E}}$, 
	then we have $f^{p}(E)\in\mathcal{E}_{per}$, $f^{p-1}(E)\notin\mathcal{E}_{per}$ by Proposition \ref{f-deg},
	and $$\deg_E(f^p)=\deg_Q(g^p).$$	 
	Thus there is a sufficiently large positive integer $N_0$ such that
	$$\deg_{\gamma_{n+p}(E)}(f^p)=\deg_E(f^p)=\deg_Q(g^p)=\deg_{\beta_{n}(Q)}(g^p)$$
	for all $n>N_0$.
	Applying Lemma \ref{f-deg}, there is an integer $N_1>0$, such that the degree $\deg_{\beta_{n}(Q)}(g^p)$ is also bounded by some constant $D$ for all $n>N_1$.
	
	{\bf Case II: Let $Q\in\widetilde{\mathcal{Q}}$ be eventually periodic.}
	In this case, $Q$ will lie in $\mathcal{Q}_{per}$ within finite iterations. Let $N_2>0$ be some integer such that the sequence 
	\begin{equation}\label{Puzzle-g}
		I(\gamma_{n+p}(Q),I(\gamma_{n+p-1}g(Q))),\cdots,I(\gamma_{n+1}(g^{p-1}(Q)))
	\end{equation}
    contain no critical points of $g$ in $V$ for all $n>N_2$.
	Since $\widetilde{\mathcal{Q}}$ is finite, the degrees of such iterations are of course bounded.
	
	\vskip0.15in
	Hence the proof is completed by taking $N=\max\{N_0,N_1,N_2\}$.
\end{proof}

\section{Bounded turning on non-trivial Julia components}\label{sec5}
The aim of this section is to prove the following:
\begin{thm}\label{uni-qc}
	Let $U$ be an infinitely connected attracting Fatou domain of a geometrically finite rational map $f$. 
	Then $U$ is conformally homeomorphic to a domain,
	whose boundary components are uniform quasicircles and points.
\end{thm}

Recall that $(f,U)$ has a holomorphic model $(g,V)$, one possible route is to prove that $V$ is a desired domain that can be uniformized. 
We now consider the shape of the non-trivial Julia components of $g$. According to \cite{C-P}, we have the following proposition.
\begin{prop}\label{wandering}
	Each wandering Julia component of a rational map with a completely invariant attracting Fatou domain is a singleton.
\end{prop}
This implies any non-trivial Julia component of $g$ is non-wandering. 

\subsection{Classification of the non-trivial Julia components}
Denote $\mathcal{Q}^{\dag}$ by the collection of non-trivial Julia components of $g$.
We have already known that each wandering Julia component of $g$ is a singleton from Proposition \ref{wandering}.
Hence all the non-trivial Julia components of $g$ are eventually periodic. Moreover, we have the following lemma.
\begin{lemma}
	The non-trivial periodic Julia components of $g$ are finite.
\end{lemma}	
\begin{proof}
	Let $Q$ be a non-trivial periodic Julia component of $g$. 
	Since the interior of $Q$ is a supperattracting Fatou domain,
	and the critical points are finite, the proof is completed. 
\end{proof}

\vskip0.15in
Let $\mathcal{Q}^{\dag}_{per}=\{Q_1,Q_2,\cdots,Q_k\}$ be the collection of non-trivial periodic Julia components of $g$, which consists of finitely many quasicircles.
We also let
$$\mathcal{Q}^{\dag}_{per}[p]=\{Q:Q\in\mathcal{Q}^{\dag},\ Q\notin\mathcal{Q}^{\dag}_{per},\ g^p(Q)\in\mathcal{Q}^{\dag}_{per},\ g^{p-1}(Q)\notin\mathcal{Q}^{\dag}_{per}\}.$$
Denote
$$d(p)=\max\{\deg(g^p:\widehat{Q}\rightarrow\widehat{g^p(Q)}):\ Q\in\mathcal{Q}^{\dag}_{per}[p]\}.$$
Clearly, $d(p)$ increases with respect to $p$. Let the depth $n$ tends to infinity, 
then the degree $d(p)$ is bounded by some constant due to Lemma \ref{g-deg}. 
Hence there is an integer $M$ such that $d(p)=d(M)$ for all $p>M$. 
Denote
$$\mathcal{Q}^*=\mathcal{Q}^{\dag}_{per}\cup\left(\displaystyle\bigcup^M_{p=1}\mathcal{Q}^{\dag}_{per}[p]\right).$$
Given $Q^{\dag}\in\mathcal{Q}^{\dag}-\mathcal{Q}^*$, 
there must be an integer $p_{\dag}$, such that
\begin{equation}
	g^{p_{\dag}}(Q^{\dag})\in\mathcal{Q},\ g^{p_{\dag}-1}(Q^{\dag})\notin\mathcal{Q}.
\end{equation}
Moveover, we have 
$g^{p_{\dag}}:\widehat{Q^{\dag}}\rightarrow\widehat{g^{p_{\dag}}(Q^{\dag})}$
is a holomorphic homeomorphism. This fact will be used in the proof of Theorem \ref{uni-qc}.
\subsection{The proofs of Theorem \ref{uni-qc} and Theorem \ref{Fatou-Koebe}}
In preparation for the proofs, we first introduce the following lemma.
\begin{lemma}\label{mainlem}
	Let $V^*\subset\mathbb{C}$ be a Jordan disk and $S^*\subset V^*$ be a quasicircle. 
	We also let $V_\alpha\subset\mathbb{C}$ be Jordan disks and $S_\alpha\subset V_\alpha$ be quasicircles, 
	where $\alpha\in\Lambda$, $\Lambda$ is some index set.
	Suppose that $f_{\alpha}:V_\alpha\rightarrow V^*$ is a holomorphic homeomorphism and $f_{\alpha}(S_\alpha)=S^*$ for each $\alpha\in\Lambda$.
	Then $\{S_\alpha: \alpha\in\Lambda\}$ are uniform quasicircles.
\end{lemma}
\begin{proof}
	For any $x,y\in S^*$, let $L(x,y)\subset S^*$ be the subarc joining $x$ and $y$ with
	$$\operatorname{diam}(L(x,y))\leq\operatorname{diam}(S^*\setminus L(x,y)).$$
	By Theorem \ref{Ahlfors}, there is a constant $K^*$, such that $\Delta(L(x,y);x,y)\leq K^*$.
	
	Similarly, for any $\alpha\in\Lambda$ and for any $x_\alpha,y_\alpha\in S_\alpha$, 
	let $L(x_\alpha,y_\alpha)\subset S_\alpha$ be the subarc joining $x_\alpha$ and $y_\alpha$ with 
	$$\operatorname{diam}(L(x_\alpha,y_\alpha))\leq\operatorname{diam}(S_\alpha\setminus L(x_\alpha,y_\alpha)).$$
	In addition, we let $\Gamma(x_\alpha,y_\alpha)=S_\alpha\setminus L(x_\alpha,y_\alpha)$. 
	Since $f_\alpha$ is homeomorphic on $S_\alpha$, the quasiarc $f_\alpha(L(x_{\alpha},y_{\alpha}))$ must be the proper subset of $S$.
	
	Denote $2m=\operatorname{mod}(V^*\setminus\widehat{S^*})$, where $\widehat{S^*}$ is the union of $S^*$ and the connected component of $\widehat{\mathbb{C}}-S^*$ that disjoint from $\infty$. There exists a Jordan domain $U^*$, such that
	$$S^*\subset U^*\subset\overline{U}^*\subset V^*,\ \ \operatorname{mod}(V^*\setminus\overline{U}^*)\geq m.$$
	Let $U_{\alpha}=f^{-1}_{\alpha}(U^*)$. 
	It is also a Jordan domain and 
	$S_{\alpha}\subset U_\alpha\subset\overline{U}_\alpha\subset V_{\alpha}.$
	To apply Theorem \ref{turning-distortion}, the proof will be broken into two cases.
	
	{\bf Case I: $\operatorname{diam}(f_\alpha(L(x_{\alpha},y_{\alpha})))
		=\operatorname{diam}(L(f_\alpha(x_{\alpha}),f_\alpha(y_{\alpha}))).$}\par
	We have
	\begin{equation}
		\begin{aligned}
			\Delta(L(x_{\alpha},y_{\alpha});x_{\alpha},y_{\alpha})\leq& D(1,m)\Delta(f_\alpha(L(x_{\alpha},y_{\alpha}));f_\alpha(x_{\alpha}),f_\alpha(y_{\alpha}))\\
			=&D(1,m)\Delta((L(f_\alpha(x_{\alpha}),f_\alpha(y_{\alpha}));f_\alpha(x_{\alpha}),f_\alpha(y_{\alpha}))\\
			\leq&D(1,m)K^*.
		\end{aligned}
	\end{equation}
	
	{\bf Case II: $\operatorname{diam}(f_\alpha(L(x_{\alpha},y_{\alpha})))\neq \operatorname{diam}(L(f_\alpha(x_{\alpha}),f_\alpha(y_{\alpha}))).$}\par
	In this case, $f_\alpha(L(x_{\alpha},y_{\alpha}))$ is the subarc with larger diameter. This implies that $f_\alpha$ maps $\Gamma(x_{\alpha},y_{\alpha})$ to the subarc of $S^*$ with smaller diameter, i.e.
	$$\operatorname{diam}(f_\alpha(\Gamma(x_{\alpha},y_{\alpha})))=\operatorname{diam}(L(f_\alpha(x_{\alpha}),f_\alpha(y_{\alpha}))).$$ 
	It then follows that
	\begin{equation}
		\begin{aligned}
			\Delta(L(x_{\alpha},y_{\alpha});x_{\alpha},y_{\alpha})
			\leq&\Delta(\Gamma(x_{\alpha},y_{\alpha});x_{\alpha},y_{\alpha})\\ \leq&D(1,m)\Delta(f_\alpha(\Gamma(x_{\alpha},y_{\alpha}));f_\alpha(x_{\alpha}),f_\alpha(y_{\alpha}))\\
			=&D(1,m)\Delta((L(f_\alpha(x_{\alpha}),f_\alpha(y_{\alpha}));f_\alpha(x_{\alpha}),f_\alpha(y_{\alpha}))\\
			\leq&D(1,m)K^*.
		\end{aligned}
	\end{equation}
	
	\vskip0.25in
	As a result of Theorem \ref{Ahlfors}, 
	we obtain that $S_\alpha$ are uniform $K-$quasicircles for all $\alpha\in\Lambda$ by taking $K=D(1,m)K^*$.
\end{proof}

\begin{proof}[Proof of Theorem \ref{uni-qc}]
	Without loss of generality, we assume that $U$ is $f-$invariant. 
	Let $(g,V)$ be a holomorphic model for $(f,U)$.
	We already known that $V=\varphi(U)$, where $\varphi$ is conformal.
	Moreover, the non-trivial Julia components of $g$ are quasicircles.	
	
	\vskip0.15in
    {\bf Step I: Bounded turning on $Q\in\mathcal{Q}^*$.}\par
    By Ahlfors' characterization of quasicircle(Theorem \ref{Ahlfors}), 
    for each $Q\in\mathcal{Q}^*$,
    there exists some constant $K_Q$ such that
    $$\Delta(L(x,y);x,y)\leq K_Q$$ 
    where $x,y$ are arbitrary two different points in $Q\in\mathcal{Q}$
    and $L(x,y)\subset Q$ is a quasiarc joining $x$ and $y$ with
    $$\operatorname{diam}(L(x,y))\leq\operatorname{diam}(Q\setminus L(x,y)).$$ 

	It is clear that $\mathcal{Q}^*$ contains finite quasicircles. 
	Then $K_1=\max\{K_Q: Q\in\mathcal{Q}^*\}$ is the uniform turning constant of all Julia components in $\mathcal{Q}^*$.

	\vskip0.15in
	{\bf Step II: Bounded turning on $Q^{\dag}\in\mathcal{Q}^{\dag}-\mathcal{Q}^*$.}\par
	For each $Q^{\dag}\in\mathcal{Q}^{\dag}-\mathcal{Q}^*$, we assume that 
	$$g^{p_{\dag}}(Q^{\dag})\in\mathcal{Q}^*,\ \ g^{p_{\dag}-1}(Q^{\dag})\notin\mathcal{Q}^*$$
	where $p_{\dag}>0$ is some integer. 
	Let $Q=g^{p_{\dag}}(Q^{\dag})$.
	Since $\deg_{\beta_{n}(Q^{\dag})}(g^{p_{\dag}})$ monotonically decreasing converges to $\deg_{Q^{\dag}}(g^{p_{\dag}})$ as $n\rightarrow\infty$, 
	there exists an integer $N_1$ such that
	\begin{equation}
		\deg(g^{p_{\dag}}:I(\beta_{n+p_{\dag}}(Q^{\dag}))\rightarrow I(\beta_{n}(Q)))=\deg(g^{p_{\dag}}:\widehat{Q^{\dag}}\rightarrow\widehat{Q})=1
	\end{equation}
	for any $n\geq N_1$. 
    This implies that 
    \begin{equation}
    	g^{p_{\dag}}:I(\beta_{n+p_{\dag}}(Q^{\dag}))\rightarrow I(\beta_{n}(Q))
    \end{equation}
	is a holomorphic homeomorphism for $n\geq N_1$.
	
	Next, we take $N_2>N_1$ and let
	$$m=\min_{Q\in\mathcal{Q}^*}\operatorname{mod}(I(\beta_{N_1}(Q))\setminus\overline{I(\beta_{N_2}(Q))}).$$
	Applying Theorem \ref{turning-distortion} to the following situation $g=g^{p_{\dag}}$ and
	\begin{equation} 
		\begin{split}
		&(V_1,U_1)=(I(\beta_{N_1+p_{\dag}}(Q^{\dag})),I(\beta_{N_2+p_{\dag}}(Q^{\dag}))),\\ &(V_2,U_2)=(I(\beta_{N_1}(Q)),I(\beta_{N_2}(Q))),
		\end{split}
	\end{equation}
	we conclude that there is a constant $D(1,m)$ independent of $p_{\dag}$ such that
	$$\Delta(L(x_{\dag},y_{\dag});x_{\dag},y_{\dag})\leq D(1,m)K_1,$$
	where $x_{\dag},y_{\dag}$ are arbitrary two different points in $Q^{\dag}$. 
	
	\vskip0.15in
	Finally, we take
	$$K=\max\{K_1,D(1,m)K_1\}.$$
	Then for each non-trivial Julia component $Q_0$ of $g$ and for each $x,y\in Q_0$, $x\neq y$, we must have
	$$\Delta(L(x,y);x,y)\leq K.$$	
	This implies that all the quasicircles in $\mathcal{Q}$ satisfy the $K-$bounded turning condition uniformly.
	The proof is completed.
\end{proof}

So far, we have demonstrated that $V$ is a cofat domain, the proof of Theorem \ref{Fatou-Koebe} is then completed by applying Theorem \ref{cofat}.


\begin{thebibliography}{99}
	\bibitem{Ahlfors}
	L.V. Ahlfors, Lectures on quasiconformal mappings,
	Univ. Lecture Ser., vol. 38, Amer. Math. Soc., 2006.
	
		\bibitem{Bers}
	L. Bers, Uniformization by Beltrami equations,
	Commun. Pur. Appl. Math., 14(3): 215-228, 1961.
	
	\bibitem{Bonk1}
	M. Bonk, Uniformization of Sierpi\'nski carpets in the plane,
	Invent. Math., 186(3): 559-665, 2011.
	
	\bibitem{Bonk2}
	M. Bonk, Uniformization by square domains,
	The Journal of Analysis, 24: 103-110, 2016.
	
	\bibitem{B-M1}
	M. Bonk and S. Merenkov, Quasisymmetric rigidity of square Sierpi\'nski carpets,
	Ann. Math., 177(2): 591-643, 2013.
	
	\bibitem{B-M2}
	M. Bonk and S. Merenkov, Square Sierpi\'nski carpets and Latt\`es maps,
	Math. Z., 296: 695-718, 2020.
	
	\bibitem{Denneberg}
	R. Denneberg, Konforme Abbildung einer Klasse unendlich vielfach zusammenh\"{a}ngender schlichter Bereiche auf Kreisbereiche, Diss,
	Leipziger Berichte, 84: 331-352, 1932.
	
	\bibitem{Grotzsch}
	H. Gr\"{o}tzsch, Eine Bemerkung zum Koebeschen Kreisnormierungsprinzip,
	Leipziger Berichte, 87: 319-324, 1935.
	
	\bibitem{Haas}
	A. Haas, Linearization and mappings onto pseudocircle domains,
	Tran. Amer. Math. Soc., 282(1): 415-429, 1984.
	
	\bibitem{H-L}
	H. Hakobyan and W. Li, Quasisymmetric embeddings of slit Sierpi\'nski carpets,
	Tran. Amer. Math. Soc., 376: 8877-8918, 2023.
	
	\bibitem{H-M}
	S. Hildebrandt and H. von der Mosel, Conformal mapping of multuply connected Riemann domains by a variational approachm,
	Adv. Calc. Var., 2: 137-183, 2009.
	
	\bibitem{H-S1}
	Z.-X. He and O. Schramm, Fixed points, Koebe uniformization and circle packings,
	Ann. Math., 137: 369-406, 1993.
	
	\bibitem{H-S2}
	Z.-X. He and O. Schramm, Rigidity of circle domains whose boundary has $\sigma-$finite linear measure,
	Invent. Math., 115(3): 297-310, 1994,.
	
	\bibitem{H-S3}
	Z.-X. He and O. Schramm, Koebe Uniformization for ``Almost Circle Domains'',
	Am. J. Math., 117(3): 653-667, 1995.
	
	\bibitem{Koebe1}
	P. Koebe, \"{U}ber die Uniformisierung beliebiger analytischer Kurven III,
	Nachr. Ges Wiss. Gott., 337-358, 1908,.
	
	\bibitem{Koebe2}
	P. Koebe, Abhandlungen zur Theorie der Konformen Abbildung: VI. Abbildung mehrfach zusammenh\"{a}ngender Bereiche auf Kreisbereiche. Uniformisierung hyperelliptischer Kurven.(Iterationsmethoden),
	Math. Z., 7: 235-301, 1920.
	
	\bibitem{Koebe3}
	P. Koebe, \"{U}ber die Konforme Abbildung Endlich- und Unendlich-Vielfach Zusammenh\"{a}ngender Symmetrischer Bereiche, 
	Acta Math., 43: 263-287, 1922.
	
	\bibitem{Meschowski1}
	H. Meschowski, 
	\"{U}ber die konformen Abbildung gewisser Bereiche von unendlich hohen Zusammenhang auf Vollkreisbereiche, I,
	Math. Ann., 123(1): 392-405, 1951.
	
	\bibitem{Meschowski2}
	H. Meschowski, 
	\"{U}ber die konformen Abbildung gewisser Bereiche von unendlich hohen Zusammenhang auf Vollkreisbereiche, II,
	Math. Ann., 124(1): 178-181, 1952.
	
	\bibitem{N-Y}
	D. Ntalampekos and M. Younsi, Rigidity theorems for circle domains,
	Invent. Math., 220(1): 129-189, 2020.
	
	\bibitem{M-W}
	S. Merenkov and K. Wildrick, Quasisymmetric Koebe uniformization,
	Rev. Mat. Iberoam., 29(3): 859–910, 2013.
	Ann. Acad. Sci Fenn-m., Series AI, 50: 1-79, 1948.
	
	\bibitem{R-R}
	K. Rajala and M. Rasimus, Quasisymmetric Koebe uniformization with weak metric doubling measures,
	Illinois J. Math. 65(): 749-767, 2021.
	
	\bibitem{Rehmert}
	J. Rehmert, Quasisymmetric Koebe uniformization of metric surfaces with uniformly relatively separated boundary,
	Ph.D. Thesis, 2022.
	
	\bibitem{Sario}
	L. Sario, \"{U}ber Riemannsche Flachen mit hebbarem Rand,
	Ann. Acad. Sci Fenn-m., Ser. AI, 50: 1-79, 1948.
	
	\bibitem{Schramm}
	O. Schramm, Transboundary extremal length,
	J. Anal. Math., 66: 307-329, 1995.	
	
	\bibitem{Sterbel1}
	K.L. Sterbel, \"{U}ber das Kreisnormierungsproblem der konformen Abbildung,
	Ann. Acad. Sci Fenn., 1: 1-22, 1951.
	
	\bibitem{Sterbel2}
	K.L. Sterbel, \"{U}ber die konformen Abbildung von Gebieten unendlich hohen Zusammenhangs,
	Comment. Math. Helv., 27: 101-127, 1953.
	
	\bibitem{Younsi}
	M. Younsi, Removability, rigidity of circle domains and Koebe's conjecture,
	Adv. Math., 303: 1300-1318, 2016.
	
	\bibitem{C-P}
	G.Z. Cui and W.J. Peng, On the structure of Fatou domains,
	Sci. China Ser. A, 51(7): 1167-1186, 2008.
	
	\bibitem{McMullen}
	C. McMullen, Complex Dynamics and Renormalization,
	Ann. of Math. Stud., vol. 135, Princeton Univ. Press, Princeton, NJ, 1994.
	
	\bibitem{Q-W-Y}
	W.Y. Qiu, X.G. Wang and Y.C. Yin, Dynamics of McMullen maps,
	Adv. Math., 229(4): 2525-2577, 2012.
	
	\bibitem{Sullivan}
	D. Sullivan, Quasiconformal homeomorphisms and dynamics I: solution of the Fatou-Julia problem on wandering domains,
	Ann. Math., 122: 401-418, 1985.
	
	\bibitem{T-Y}
	L. Tan and Y.C. Yin, Local connectivity of the Julia set for geometrically finite rational maps,
	Sci. China Ser. A, 39: 39-47, 1996.
\end{thebibliography}
\end{document}